\documentclass[12pt]{article}   	% use "amsart" instead of "article" for AMSLaTeX format, use titlepage for title on separate page
\usepackage{geometry}                		% See geometry.pdf to learn the layout options. There are lots.
\usepackage{graphicx}				% Use pdf, png, jpg, or eps§ with pdflatex; use eps in DVI mode
\usepackage{float}					% Allows [H] placement specifier with graphicx package for including figures
\usepackage{cancel}
\usepackage{centernot}
\usepackage{bbm}
\usepackage{listings}
\usepackage{fancyvrb}
\usepackage{hyperref}
\usepackage{textgreek}
\usepackage{hhline}
\usepackage[usenames,dvipsnames]{xcolor}
\usepackage{tkz-graph}
\usetikzlibrary{decorations.shapes}
\usetikzlibrary{positioning}
\tikzset{main node/.style={circle,fill=blue!20,draw,minimum size=1cm,inner sep=0pt},}
\usepackage{asymptote}
\usepackage{amssymb,amsmath,amsfonts,amsthm}
\usepackage{setspace}
\usepackage{booktabs}
\usepackage{titling}%,titlesec}
\usepackage{hyperref}
\usepackage{caption}

\makeatletter
\def\thmhead@plain#1#2#3{%
  \thmname{#1}\thmnumber{\@ifnotempty{#1}{ }\@upn{#2}}%
  \thmnote{ {\the\thm@notefont#3}}}
\let\thmhead\thmhead@plain
\makeatother

\newtheorem{theorem}{Theorem}[section]
\newtheorem{lemma}[theorem]{Lemma}
\newtheorem{proposition}[theorem]{Proposition}
\newtheorem{corollary}{Corollary}[theorem]

\newtheorem{conjecture}{Conjecture}[section]

\DeclareMathOperator{\spn}{span}

\title{New Classes of Set-Sequential Trees}
\author{Louis Golowich\thanks{MIT PRIMES, Department of Mathematics, MIT, 77 Massachusetts Ave., Cambridge, MA 02139. Email:~{\tt louis.golowich@gmail.com}.} \and 
Chiheon Kim\thanks{Department of Mathematics, MIT, 77 Massachusetts Ave., Cambridge, MA 02139. Email: {\tt chiheonk@math.mit.edu}.}}
\date{\today}							% Activate to display a given date or no date

\begin{document}
\maketitle
\begin{abstract}
A graph is called {\it set-sequential} if its vertices can be labeled with distinct nonzero vectors in $\mathbb{F}_2^n$ such that when each edge is labeled with the sum$\pmod{2}$ of its vertices, every nonzero vector in $\mathbb{F}_2^n$ is the label for either a single vertex or a single edge. We resolve certain cases of a conjecture of Balister, Gy\H{o}ri, and Schelp in order to show many new classes of trees to be set-sequential. We show that all caterpillars $T$ of diameter $k$ such that $k \leq 18$ or $|V(T)| \geq 2^{k-1}$ are set-sequential, where $T$ has only odd-degree vertices and $|T| = 2^{n-1}$ for some positive integer $n$. We also present a new method of recursively constructing set-sequential trees.
\end{abstract}

\section{Introduction}
%The field of graph labeling is based on the assignment of the vertices and edges of a graph to some set of labels. Rosa ADD CITATION introduced graceful labelings, the origin of most problems in graph labeling. A labeling of the vertices and the edges of a graph $G$ is graceful if all vertices are distinctly labeled with integers between 0 and $|E(G)|$ inclusive such that when each edge is labeled with the absolute difference of its vertices, all edge labels are distinct. Much work on graceful graphs has been motivated by the Graceful Tree Conjecture ADD CITATION, which states that all trees are graceful. ADD APPLICATIONS
%
%We consider a labeling similar to graceful labelings, in which we use vectors in $\mathbb{F}_2^n$ instead of integers to label the vertices and the edges.
A labeling of a graph $G$ with a set $S$ of labels is any function from the vertices and edges of $G$ to $S$. A graph is called {\it set-sequential} if there exists a labeling of its vertices with distinct nonzero vectors in $\mathbb{F}_2^n$ such that when each edge is labeled with the sum$\pmod{2}$ of its vertices, every nonzero vector in $\mathbb{F}_2^n$ is used exactly once as a label for either a vertex or for an edge; this definition is generally attributed to \cite{acharya_set-valuations_1983}.
%Note that because addition$\pmod{2}$ and subtraction$\pmod{2}$ and equivalent, this method of labeling edges is analogous to that of a graceful labeling, except that vectors in $\mathbb{F}_2^n$ are used as opposed to integers.
As a direct consequence of this definition, a graph $G$ can be set-sequential only if $|V(G)| + |E(G)| = 2^n - 1$ for some $n$.

Much work on the problem of classifying set-sequential graphs has been focused on trees. Of particular interest are caterpillar trees, defined to be any tree containing some path, called the center path, from which every vertex has distance at most 1. It is often useful to classify caterpillars by the diameter, which is the largest distance between any two vertices in a connected graph. In the case of caterpillars, the diameter gives the length of the center path.

Abhishek and Agustine \cite{abhishek_set-valued_2012} showed that all caterpillars of diameter at most 4 with only odd-degree vertices, and with $2^{n-1}$ vertices for some $n$, are set-sequential. Abhishek \cite{abhishek_set-valued_2013} extended this result to caterpillars of diameter 5. Mehta and Vijayakumar \cite{mehta_note_2008} showed that all paths of length $2^{n-1}$ are set-sequential if $n \geq 5$. Hegde \cite{hegde_set_2009} showed that no graph with exactly 1 or 2 vertices of even degree is set-sequential. However, no similar restriction is known for graphs with only odd-degree vertices, a fact motivating in the following conjecture. (This conjecture appears to have been proposed before, but we were unable to find a citation, so we restate it here.)
\begin{conjecture}
\label{conj:odd_deg_trees}
All trees with only odd-degree vertices and with $2^n$ vertices for some integer $n$ are set-sequential.
\end{conjecture}
Our results in this paper mark progress towards Conjecture~\ref{conj:odd_deg_trees}, and thus we mostly focus on trees with only odd-degree vertices.

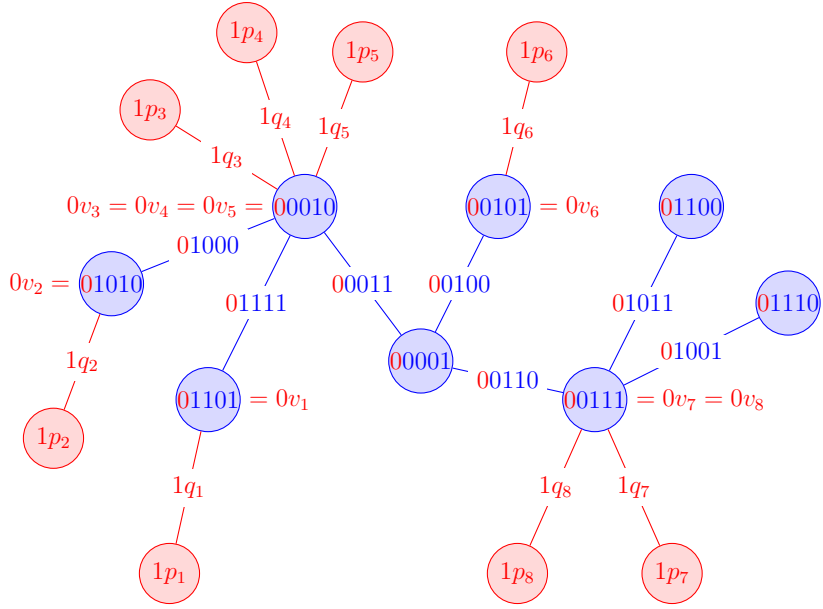
\begin{figure}
\centering
\begin{tikzpicture}[scale=2.57, every node/.style={scale=.8}]
\node[main node, fill=blue!15, draw=blue, text=blue] (1) at (.1, .2) {{\color{red}0}{\color{blue}0001}};
\node[main node, fill=blue!15, draw=blue, text=blue, label={[red]right:$=0v_7=0v_8$}] (3) at (1, 0) {{\color{red}0}{\color{blue}0111}};
\node[main node, fill=blue!15, draw=blue, text=blue, label={[red]right:$=0v_1$}] (4) at (-1, 0) {{\color{red}0}{\color{blue}1101}};
\node[main node, fill=blue!15, draw=blue, text=blue, label={[red]left:$0v_3=0v_4=0v_5=$}] (5) at (-.5, 1) {{\color{red}0}{\color{blue}0010}};
\node[main node, fill=blue!15, draw=blue, text=blue, label={[red]right:$=0v_6$}] (6) at (.5, 1) {{\color{red}0}{\color{blue}0101}};
\node[main node, fill=blue!15, draw=blue, text=blue] (7) at (1.5, 1) {{\color{red}0}{\color{blue}1100}};
\node[main node, fill=blue!15, draw=blue, text=blue] (8) at (2, .5) {{\color{red}0}{\color{blue}1110}};
\node[main node, fill=blue!15, draw=blue, text=blue, label={[red]left:$0v_2=$}] (9) at (-1.5, .6) {{\color{red}0}{\color{blue}1010}};

\draw[blue] (1) -- (3) node[midway, fill=white] {{\color{red}0}{\color{blue}0110}};
\draw[blue] (1) -- (5) node[midway, fill=white] {{\color{red}0}{\color{blue}0011}};
\draw[blue] (5) -- (4) node[midway, fill=white] {{\color{red}0}{\color{blue}1111}};
\draw[blue] (5) -- (9) node[midway, fill=white] {{\color{red}0}{\color{blue}1000}};
\draw[blue] (1) -- (6) node[midway, fill=white] {{\color{red}0}{\color{blue}0100}};
\draw[blue] (3) -- (7) node[midway, fill=white] {{\color{red}0}{\color{blue}1011}};
\draw[blue] (3) -- (8) node[midway, fill=white] {{\color{red}0}{\color{blue}1001}};

\node[main node, fill=red!15, draw=red, text=red] (10) at (-1.2, -.9) {$1p_1$};
\node[main node, fill=red!15, draw=red, text=red] (11) at (-1.8, -.2) {$1p_2$};
\node[main node, fill=red!15, draw=red, text=red] (12) at (-1.3, 1.5) {1$p_3$};
\node[main node, fill=red!15, draw=red, text=red] (13) at (-.8, 1.9) {$1p_4$};
\node[main node, fill=red!15, draw=red, text=red] (14) at (-.2, 1.8) {$1p_5$};
\node[main node, fill=red!15, draw=red, text=red] (15) at (.7, 1.8) {$1p_6$};
\node[main node, fill=red!15, draw=red, text=red] (16) at (.6, -.9) {$1p_8$};
\node[main node, fill=red!15, draw=red, text=red] (17) at (1.4, -.9) {$1p_7$};

\draw[red] (4) -- (10) node[midway, fill=white] {$1q_1$};
\draw[red] (9) -- (11) node[midway, fill=white] {$1q_2$};
\draw[red] (5) -- (12) node[midway, fill=white] {$1q_3$};
\draw[red] (5) -- (13) node[midway, fill=white] {$1q_4$};
\draw[red] (5) -- (14) node[midway, fill=white] {$1q_5$};
\draw[red] (6) -- (15) node[midway, fill=white] {$1q_6$};
\draw[red] (3) -- (16) node[midway, fill=white] {$1q_8$};
\draw[red] (3) -- (17) node[midway, fill=white] {$1q_7$};
\end{tikzpicture}
\caption{An illustration of the method for generating set-sequential trees introduced in \cite{balister_coloring_2011} that was presented as motivation for Conjecture~\ref{conj:main_conj}. The set-sequential labeling of the tree $T$ with $2^3$ vertices, shown in blue, is used to generate a set-sequential labeling of the tree $T'$ with $2^4$ vertices by adding the $2^3$ pendant edges shown in red. By definition, the $p_i$'s and $q_i$'s together must cover all vectors in $\mathbb{F}_2^4$, and must satisfy $p_i + q_i = v_i$ for all $i$.}
\label{fig:add_pends}
\end{figure}

Balister, Gy\H{o}ri, and Schelp \cite{balister_coloring_2011} proposed a general method to generate set-sequential trees by adding $2^{n-1}$ pendant edges to a set-sequential tree with $2^{n-1}$ vertices. Specifically, consider a set-sequential tree $T$ with $2^{n-1}$ vertices. Then define $T'$ to be a tree with $2^n$ vertices that is obtained by adding $2^{n-1}$ pendant edges to $T$, and let the $i$th new pendant edge for $1 \leq i \leq 2^{n-1}$ be attached to the vertex in $T$ labeled $v_i$. The following structure for a set-sequential labeling of $T'$ was proposed in \cite{balister_coloring_2011}, and is illustrated in Figure~\ref{fig:add_pends}. Append $0$ to the beginning of the label of every vertex and edge that was already in $T$ (every vertex and edge that is not from a new pendant edge), and let the $i$th new pendant edge consist of a vertex labeled $1p_i$ and an edge $1q_i$, where $ab$ denotes the concatenation of $a$ and $b$. Then for $T'$ to be set-sequential, the following two conditions must hold by definition:
\begin{itemize}
\item The set of all $p_i$'s and $q_i$'s for $1 \leq i \leq 2^{n-1}$ consists of all vectors in $\mathbb{F}_2^n$.
\item For all $1 \leq i \leq 2^{n-1}$, it holds that $1p_i + 0v_i = 1q_i$, or equivalently, $p_i + q_i = v_i$.
\end{itemize}
It follows by these constraints that if $n \geq 2$, then 
\begin{equation}
\label{eq:vi_sum_0}
\sum_{i=1}^{2^{n-1}} v_i = \sum_{i=1}^{2^{n-1}} p_i + \sum_{i=1}^{2^{n-1}} q_i = \sum_{v \in \mathbb{F}_2^n} v = 0.
\end{equation}
It was conjectured in \cite{balister_coloring_2011} that for any choice of the $v_i$'s satisfying (\ref{eq:vi_sum_0}), there always exists a choice of the $p_i$'s and $q_i$'s for which the two conditions above hold, as formalized below.

%Here, adding a pendant edge consists of adding a new vertex to the graph, and adding a new edge connecting the new vertex to one of the original vertices. If the new vertices are labeled $p_1, \dots, p_{2^{n-1}}$, and are connected by the new edges labeled $q_1, \dots, q_{2^{n-1}}$ to the original vertices labeled $v_1, \dots, v_{2^{n-1}}$ respectively, then it is necessary for $p_i + q_i = v_i$ for all $1 \leq i \leq 2^{n-1}$. The authors applied this idea to formulate the following conjecture.
\begin{conjecture}[\cite{balister_coloring_2011}]
\label{conj:main_conj}
For any $2^{n-1}$ non-zero vectors $v_1, \dots, v_{2^{n-1}} \in \mathbb{F}_2^n$ with $n \geq 2$ and $\sum_{i=1}^{2^{n-1}} v_i = 0$, there exists a partition of $\mathbb{F}_2^n$ into pairs of vectors $(p_i, q_i)$ for $1 \leq i \leq 2^{n-1}$ such that $v_i = p_i + q_i$ for all $i$.
\end{conjecture}
Conjecture \ref{conj:main_conj} was shown in \cite{balister_coloring_2011} to hold if $n \leq 5$ or if $v_1 = \dots = v_{2^{n-2}}$ and $v_{2i-1} = v_{2i}$ for all $1 \leq i \leq 2^{n-2}$.

In this paper, we resolve additional cases of Conjecture \ref{conj:main_conj} in order to show the set-sequentialness of new classes of trees. We resolve the conjecture in cases where either the dimension of the span of the vectors $v_1, \dots, v_{2^{n-1}}$ or the number of distinct vectors $v_i$ is restricted. We then show that all caterpillars $T$ with only odd-degree vertices of diameter $k$ such that $k \leq 18$ or $|V(T)| \geq 2^{k-1}$ are set-sequential, assuming that $T$ has $2^{n-1}$ vertices for some positive integer $n$. We also present a new method of constructing a set-sequential tree by connecting four copies of an existing set-sequential tree, which enables us to prove the set-sequentialness of a broader class of trees that is not covered by the method presented in~\cite{balister_coloring_2011}.

The organization of this paper is as follows. In Section \ref{sec:conj_progress}, we partially prove Conjecture~\ref{conj:main_conj} by placing additional restrictions on the dimension of the span of $v_1, \dots, v_{2^{n-1}}$, or on the number of distinct vectors $v_i$. In Section \ref{subsec:app_conj}, we apply the results from Section \ref{sec:conj_progress} to showing new classes of trees, and specifically caterpillars, to be set-sequential. In Section \ref{subsec:copies_induction}, we present an alternative method of constructing set-sequential trees. We provide concluding remarks in Section \ref{sec:conclusion}.

\section{Progress towards Conjecture \ref{conj:main_conj}}
\label{sec:conj_progress}

We begin with the case where $\dim(\spn\{v_1, \dots, v_{2^{n-1}}\}) \leq 5$, utilizing the fact that the conjecture is already known to hold when $n$ is at most $5$. We first need the following lemma.

\begin{lemma}
\label{lem:partition}
Let $n \geq 3$ be an integer. For any $2^{n-1}$ nonzero vectors $v_1, \dots, v_{2^{n-1}} \in \mathbb{F}_2^n$ with sum 0 such that their span has dimension less than $n$, there exists a partition of $\{v_1, \dots, v_{2^{n-1}}\}$ into two subsets of equal size such that the sum of the elements in each subset is 0.
\end{lemma}
\begin{proof}
% ONLY NEED A; B FOLLOWS; ALSO DONT DEFINE, AS THEY ARE CONSTRUCTED
We provide an algorithm for partitioning $v_1, \dots, v_{2^{n-1}}$ into subsets $A$ and $B$, each of size $2^{n-2}$. If $n \leq 5$, let $S$ be some $(n-1)$-dimensional subspace of $\mathbb{F}_2^n$ such that $\spn\{v_1, \dots, v_{2^{n-1}}\}$ is a subspace of $S$, and let $S^\complement = \mathbb{F}_2^n - S$ be the $(n-1)$-dimensional affine subspace consisting of all vectors in $\mathbb{F}_2^n$ in the complement of $S$. Because Conjecture~\ref{conj:main_conj} is known to hold, there exists a partition of $\mathbb{F}_2^n$ into pairs $(p_i, q_i)$ such that $p_i + q_i = v_i$ for all $1 \leq i \leq 2^{n-1}$. Furthermore, by assumption, $p_i$ and $q_i$ must either both be in $S$ or both be in $S^\complement$ for all $1 \leq i \leq 2^{n-1}$; if this were not true, then $v_i = p_i + q_i$ could not belong to $S$, a contradiction. Therefore it is sufficient to let $A = \{v_i : p_i, q_i \in S\}$ and $B = \{v_i : p_i, q_i \in S^\complement\}$, as then the sum of all elements in $A$ and $B$ is equal to the sum of all elements in $S$ and $S^\complement$ respectively, which is $0$ as $S$ and $S^\complement$ are affine subspaces of dimension greater than 1. % NOT NEEDED FOR <= 5

Now assume $n \geq 6$. Assume without loss of generality that $v_1, \dots, v_l$ are the only distinct vectors occurring an odd number of times over all vectors $v_i$. Therefore $\sum_{i=0}^l v_i = 0$, and $l$ must be even, so also assume that $v_{2i-1} = v_{2i}$ for all $i > l/2$. We first present a method of partitioning $v_1, \dots, v_l$ into $A$ and $B$ such that neither subset has more than $2^{n-2}$ elements, each subset has an even number of elements, and the sum of the elements in each subset is 0. If $l \leq 2^{n-2}$, simply place all $v_1, \dots, v_l$ into $A$. Otherwise, note the following:
\begin{itemize}
\item For any $n$-vector subset $V$ of $\{v_1, \dots, v_l\}$, there exists some nonempty subset of $V$ having sum 0, as $\dim(\spn\{v_1, \dots, v_l\}) < n$.
\item It holds that $l \leq 2^{n-1} - 2$, as if $l = 2^{n-1}$ then $v_1, \dots, v_l$ form a complete $(n-1)$-dimensional subspace of $\mathbb{F}_2^n$, implying that some $v_i$ is equal to 0, which is a contradiction.
\end{itemize}
We use these observations to deal with the cases where $n = 6$ and $n > 6$ separately:
\begin{description}
\item[$n > 6$.] If $2^{n-2} < l \leq 2^{n-1} - 2n + 1$, then begin by placing all vectors $v_1, \dots, v_l$ in $A$. Then while $|A| > 2^{n-2}$, there must exist a nonempty subset $V$ of at most $n$ vectors in $A$ with a sum of zero; transfer all vectors in $V$ from $A$ to $B$, and repeat the process. If $|A|$ is odd at the end of this process, there must be a subset $V$ of at most $n$ vectors in $A$ with a sum of zero such that $|V|$ is odd; transfer all vectors in $V$ from $A$ to $B$, completing the partition.

If $2^{n-1} - 2n + 1 < l \leq 2^{n-1} - 2$, then by the pigeonhole principle there must be at least $\lceil n/2 \rceil$ disjoint 4-vector subsets of $\{v_1, \dots, v_l\}$ with sum 0; initially place these subsets in $B$, and all other $v_i$'s for $1 \leq i \leq l$ in $A$. Specifically, these 4-vector subsets with sum 0 can consist of all 4-vector subsets of $\{v_1, \dots, v_l\}$ of the form $\{x00, x01, x10, x11\}$, where $x$ is any $(n-2)$-length prefix, and the basis is chosen such that the first component of all $v_i$'s is $0$. Then carry out the procedure used for the case where $2^{n-2} < l \leq 2^{n-1} - 2n + 1$, but transfer sufficiently many of the 4-element subsets back to $A$ so that $|A| = 2^{n-2} - 2$ or $2^{n-2}$.

\item[$n = 6$.] If $26 \leq l \leq 30$, simply apply the method used for the case where $n > 6$. % Techically for l = 26 there are only 2, not 3 4-vector subsets, but you only need 2; the worst case is that |A| = 17 - 6 - 5 = 6 after the first step, so then adding 8 vectors back in gives |A| = 6 + 8 = 14.
If $18 \leq l \leq 22$, first place all vectors $v_1, \dots, v_l$ in $A$. Then there must be some subset $V$ of 6, 8, or 10 vectors in $A$ with sum 0, so simply transfer the vectors in $V$ from $A$ to $B$ (The existence of $V$ can be seen by repeatedly applying the first observation above in order to partition $v_1, \dots, v_l$ into subsets of 3, 4, 5, or 6 vectors, each with sum 0). Otherwise $l = 24$ (as $l$ is even), so again place $v_1, \dots, v_l$ in $A$. It follows that there must be some subset $V$ of 8, 10, or 12 vectors in $A$ with sum 0, so transfer the vectors in $V$ from $A$ to $B$.
\end{description}

Then the partition of all the vectors is easily completed by partitioning the remaining vectors into $A$ and $B$ in any way such that $|A| = |B| = 2^{n-2}$ and $v_{2i-1}$ and $v_{2i}$ are either both in $A$ or both in $B$ for all $i > l/2$.
\end{proof}

Below, we resolve the case of Conjecture~\ref{conj:main_conj} for which $\dim(\spn\{v_1, \dots, v_{2^{n-1}}\}) \leq k$ if Conjecture \ref{conj:main_conj} is known to hold in $\mathbb{F}_2^k$.
\begin{proposition}
\label{thm:small_dimension}
If Conjecture \ref{conj:main_conj} holds in $\mathbb{F}_2^k$, then for any choice of $2^{n-1}$ nonzero vectors $v_1, \dots, v_{2^{n-1}} \in \mathbb{F}_2^n$ where $n \geq k$ and $\sum_{i=1}^{2^{n-1}} v_i = 0$ such that $\dim(\spn\{v_1, \dots, v_{2^{n-1}}\}) \leq k$, there exists a partition of $\mathbb{F}_2^n$ into pairs $(p_i, q_i)$ such that $p_i + q_i = v_i$ for all $1 \leq i \leq 2^{n-1}$. %(or loosely speaking, Conjecture \ref{conj:main_conj} holds for $v_1, \dots, v_{2^{n-1}}$).
\end{proposition}
\begin{proof}
Let $S$ be some $k$-dimensional subspace of $\mathbb{F}_2^n$ such that $\spn\{v_1, \dots, v_{2^{n-1}}\}$ is a subspace of $S$. Partition the $2^n$ vectors in $\mathbb{F}_2^n$ into $2^{n-k}$ distinct affine subspaces $S_1, \dots, S_{2^{n-k}}$ obtained by translating $S$ by translation vectors $t_1, \dots, t_{2^{n-k}}$ respectively; $S_i$ consists of each vector in $S$ added to $t_i$. Likewise, partition $\{v_1, \dots, v_{2^{n-1}}\}$ into $2^{n-k}$ subsets $V_1, \dots, V_{2^{n-k}}$, each of size $2^{k-1}$ and with sum 0; such a partition can be constructed by recursively applying Lemma \ref{lem:partition}. Because Conjecture \ref{conj:main_conj} holds in $\mathbb{F}_2^k$ by assumption, for each $V_i = \{v_{i,1}, \dots, v_{i,2^{k-1}}\}$ there exists a partition of $S$ into pairs of vectors $(p_{i,j}, q_{i,j})$ for $1 \leq j \leq 2^{k-1}$ such that $p_{i,j} + q_{i,j} = v_{i,j}$ for all $j$. Then let $(p_{i,j}', q_{i,j}') = (p_{i,j} + t_i, q_{i,j} + t_i)$. It follows that the pairs $(p_{i,j}', q_{i,j}')$ form a partition of the elements of $S_i$, and $p_{i,j}' + q_{i,j}' = v_{i,j}$ for all $j$. Therefore, as the set of all vectors $v_{i,j}$ is equivalent to the set of all vectors $v_i$, and all $S_i$ are disjoint, the pairs $(p_{i,j}', q_{i,j}')$ give a desired partition of $\mathbb{F}_2^n$.
\end{proof}

Because Conjecture~\ref{conj:main_conj} was shown in \cite{balister_coloring_2011} to hold in $\mathbb{F}_2^2, \dots, \mathbb{F}_2^5$, the following corollary directly follows from Proposition~\ref{thm:small_dimension}.
\begin{corollary}
\label{cor:dim_5}
For any $2^{n-1}$ nonzero vectors $v_1, \dots, v_{2^{n-1}} \in \mathbb{F}_2^n$ such that $\sum_{i=1}^{2^{n-1}} v_i = 0$ and $\dim(\spn\{v_1, \dots, v_{2^{n-1}}\}) \leq 5$, there exists a partition of $\mathbb{F}_2^n$ into pairs $(p_i, q_i)$ such that $p_i + q_i = v_i$ for all $1 \leq i \leq 2^{n-1}$.
\end{corollary}

%\begin{corollary}
%\label{cor:add_to_5}
%For any set-sequential graph $G$ with $2^{n-1}$ vertices, let $H$ be any graph obtained by choosing $l$ vertices in $G$, $1 \leq l \leq 5$, and adding a total of $2^{n-1}$ pendant edges to those $l$ vertices such that each vertex receives an even number of pendant edges. Then $H$ is set-sequential.
%\end{corollary}
%\begin{proof}
%Let $v_1, \dots, v_{2^{n-1}}$ be the labels of the vertices attached to the added pendant edges for some set-sequential labeling of $G$, and let $k = \dim(\spn\{v_1, \dots, v_{2^{n-1}}\})$, so that $k \leq 5$. Because each vertex receives an even number of pendant edges, we can assume that $v_{2i - 1} = v_{2i}$ for all $i$. Then the result follows directly from the theorem, because Conjecture \ref{conj:main_conj} is know to be true for $\mathbb{F}_2^k$ for $2 \leq k \leq 5$, and the case where all vertices are attached to the same vertex ($k = 1$) is trivial.
%\end{proof}

As a result of Corollary \ref{cor:dim_5}, set-sequential trees can easily be constructed in the following manner. For any given $2^{n-1}$-vertex set-sequential tree $T$, consider any tree $T'$ obtained by adding a total of $2^{n-1}$ pendant edges to at most 5 distinct vertices in $T$, such that each vertex receives an even number of new pendant edges. Then $T'$ must be a set-sequential $2^n$-vertex tree. By using stars and caterpillars of diameter 4 or 5, which are known to be set-sequential, as base cases, this method already provides a way of inductively showing the set-sequentialness of infinitely many trees. We explore this idea in more detail in Section~\ref{subsec:app_conj}.

%\begin{remark}
%Note that by making every $v_i$ occur an even number of times, the condition that the $v_i$ can be partitioned into subsets of size $2^{k-1}$ with sum $0$ is easily satisfied. Also note that doing so means that pendant edges can be added to a set-sequential graph (with an even number added to each vertex) such that the new graph is guaranteed to be set-sequential, without knowing the actual labeling of the original graph. Also, this condition is sufficient for using induction to show that large classes of graphs with only odd-degree vertices are set sequential. However, if Conjecture \ref{conj:main_conj} is true, then it is easy to see that such a partition must always exist for any valid choice of the $v_i$, as for each affine subspace $S_j$, if $p_i \in S_j$, then $q_i = p_i + v_i$ must also be in $S_j$ by definition. Therefore we can let $V_j = \{v_i : p_i, q_i \in S_j\}$.
%\end{remark}

%One interesting application of this theorem is to show that extremely large classes of caterpillars are set-sequential. For example, for all $d \leq 5$, there exists a finite set of caterpillars with a $d$-vertex path such that if all of these caterpillars are set-sequential, then all caterpillars with a $d$-vertex path and with each vertex having odd degree are set sequential. This finite set consists of all caterpillars $G$ with a $d$-vertex path and all vertices having odd degree such that $|V(G)| = 2^{\lceil \log_2 2d - 2 \rceil}$.

% GRAMMAR: v_i's, V_j from Theorem, etc.
In many cases, when partitioning the $v_i$'s into the subsets $V_j$ as in the proof of Proposition~\ref{thm:small_dimension}, there exists a partition for which the span of the elements of each subset has dimension less than the dimension of the span of all the $v_i$'s. This idea is applied in the result below, which proves another restricted case of Conjecture~\ref{conj:main_conj}.

\begin{lemma}
\label{lem:n_min_1_values}
Let $v_1, \dots, v_{2^{n-1}} \in \mathbb{F}_2^n$ satisfy $v_{2i-1} = v_{2i}$ for all $1 \leq i \leq 2^{n-2}$ and let there be at most $l < n$ distinct vectors $u_1, \dots, u_l$ among the $v_i$'s, so that each $v_i$ is equal to some value in $u_1, \dots, u_l$. Then there exists a partition of $\mathbb{F}_2^n$ into pairs $(p_i, q_i)$ such that $p_i + q_i = v_i$ for all $1 \leq i \leq 2^{n-1}$.
\end{lemma}
\begin{proof}
We show the result through induction on $n$. For the base case, Conjecture \ref{conj:main_conj} is true for $\mathbb{F}_2^2$. For the inductive step, assume that the result holds for all choices of the vectors $v_i$ in $\mathbb{F}_2^{n-1}$. Let $v_1, \dots, v_{2^{n-1}} \in \mathbb{F}_2^n$ be any vectors such that $v_{2i - 1} = v_{2i}$ and such that each $v_i$ is equal to some $u_1, \dots, u_l$, where $l < n$. Let $S$ be any $(n-1)$-dimensional subspace of $\mathbb{F}_2^n$ such that $\spn\{v_1, \dots, v_{2^{n-1}}\}$ is a subspace of $S$. If $l \leq 2$, we are done by Corollary \ref{cor:dim_5}.

Otherwise, if $l \geq 3$, assume without loss of generality that neither $u_1$ nor $u_2$ occur more than $2^{n-2}$ times among the $v_i$'s. Place all copies of $u_1$ into $V_1$ and all copies of $u_2$ into $V_2$, then distribute the rest of the $v_i$'s into $V_1$ and $V_2$ in any way such that each vector still appears an even number of times in $V_1$ and in $V_2$, and such that $|V_1| = |V_2| = 2^{n-2}$. Then let $S_1 = S$ and $S_2$ be the unique affine subspace obtained by translating $S_1$. Then by the inductive hypothesis the elements of $S_1$ and $S_2$ can be partitioned into pairs summing to the elements of $V_1$ and $V_2$ respectively, as both $V_1$ and $V_2$ contain at most $n-2$ distinct vectors by construction.
\end{proof}

Although the condition of having an even number of copies of each vector may seem arbitrary, it is in fact motivated by set-sequential trees. We mostly focus on trees with only odd-degree vertices due to Conjecture~\ref{conj:odd_deg_trees}, and therefore adding an odd number of pendant edges to a vertex, which corresponds to having an odd number of copies of some vector $v_i$, is often not useful. Furthermore, as the sum of a vector and itself is 0 in $\mathbb{F}_2^n$, the condition that $\sum_{i=1}^{2^{n-1}} v_i = 0$ is automatically met, so we can use our progress on Conjecture \ref{conj:main_conj} to inductively produce set-sequential trees without having to keep track of the actual labels of the vertices.

We now improve the result in Corollary \ref{cor:dim_5} by showing that Conjecture \ref{conj:main_conj} holds if each vector occurs an even number of times and if $\dim(\spn\{v_1, \dots, v_{2^{n-1}}\}) \leq n/2$, thereby improving a constant bound on the dimension of the span of the vectors to a linear bound.

\begin{proposition}
\label{lem:dim_span_n_div_2}
For any set of $2^{n-1}$ nonzero vectors $v_1, \dots, v_{2^{n-1}} \in \mathbb{F}_2^n$ such that $v_{2i - 1} = v_{2i}$ for all $1 \leq i \leq 2^{n-2}$ and $\dim(\spn\{v_1, \dots, v_{2^{n-1}}\}) \leq n/2$, there exists a partition of $\mathbb{F}_2^n$ into pairs $(p_i, q_i)$ such that $p_i + q_i = v_i$ for all $1 \leq i \leq 2^{n-1}$.
\end{proposition}

\begin{proof}
Let $S = \spn\{v_1, \dots, v_{2^{n-1}}\}$ and $k = \dim(S)$. We now present an algorithm to partition $v_1, \dots, v_{2^{n-1}}$ into $2^k$ subsets of size $2^{n-1-k}$, each containing at most 3 distinct values, and with an even number of copies of each value. At the beginning, there is a single subset, consisting of all vectors $\{v_1, \dots, v_{2^{n-1}}\}$. The algorithm consists of $k$ steps, during each of which every existing subset is partitioned into two subsets of equal size in the following manner.

% FIRST PUT ALGORITHM (PSEUDOCODE?), THEN PROVE THAT IT GET DOWN TO AT MOST 3 DISTINCT VALUES, WITHOUT UGLY EQUATION
Let $S_0$ be some subset at some point in the algorithm, and assume it is partitioned into $S_1$ and $S_2$. Let the distinct values in $S_0$ be $u_1, \dots, u_l$, and assume without loss of generality that they are sorted in increasing order of occurrence, so that if $u_i$ occurs fewer times than $u_j$ in $S_0$, then $i < j$. Then let $S_1$ contain all copies of $u_i$ in $S_0$ for odd $i$ and let $S_2$ contain all copies of $u_i$ in $S_0$ for even $i$. If $|S_1| > |S_2|$, then transfer $(|S_1| - |S_2|)/2$ copies of $u_l$ from $S_1$ back to $S_2$; perform the opposite operation if $|S_1| < |S_2|$.

It follows by construction that for any $S_0$ with $l$ distinct values, both $S_1$ and $S_2$ contain at most $l/2 + 1$ distinct vectors. Then, by induction, as there are at most $2^k$ distinct vectors in $\{v_1, \dots, v_{2^{n-1}}\}$, the number of distinct elements of some subset $S_0$ after $t$ steps of the algorithm is at most $$\frac{2^k}{2^t} + \sum_{x=0}^{t-1} \frac{1}{2^x},$$ so after all $k$ steps, the final number of distinct elements in each subset is at most 
\begin{equation}
%\textrm{$k$ levels} \left \{ \hspace{.4cm} \vcenter{\hbox{$
%\dfrac{
%\begin{matrix}
%\dfrac{\dfrac{\dfrac{2^k}{2} + 1}{2} + 1}{2}  + 1 & \\
%\hspace{2.2cm} \ddots & \\
%\end{matrix}
%}{2} + 1
%$}} \right.
%= 
\frac{2^k}{2^k} + \sum_{x = 0}^{k-1} \frac{1}{2^x} \leq 1 + 2 = 3.
\nonumber
\end{equation}

By assumption $\dim(S) = k \leq n - k$, so there exists a partition of $\mathbb{F}_2^n$ into $2^k$ affine subspaces obtained by translating some $(n-k)$-dimensional subspace containing $S$. These affine subspaces can then be paired with the subsets of $\{v_1, \dots, v_{2^{n-1}}\}$ so the problem is reduced to showing that Conjecture \ref{conj:main_conj} is true where there are at most 3 distinct vectors $v_i$. This result follows by Corollary \ref{cor:dim_5}.
\end{proof}
%GET RID OF EVEN COPIES CONDITION IN ABOVE PROOF IF POSSIBLE

%\subsection{Reducing Conjecture \ref{conj:main_conj} to smaller $n$ when the dimension of the span of $v_1, \dots, v_{2^{n-1}}$ is equal to $n$}

Now we prove a relatively restricted case of Conjecture \ref{conj:main_conj}, which can be used to prove the set-sequentialness of many small trees. We also generalize our proof in additional results below.
\begin{proposition}
\label{lem:1_more_even}
If Conjecture~\ref{conj:main_conj} holds in $\mathbb{F}_2^{n-1}$, then for any $v_1, \dots, v_{2^n-1} \in \mathbb{F}_2^n$ such that $v_{2i-1} = v_{2i}$ for all $1 \leq i \leq 2^{n-1}$, there exists a partition of $\mathbb{F}_2^n$ into pairs $(p_i, q_i)$ such that $p_i + q_i = v_i$ for all $1 \leq i \leq 2^{n-1}$.
% Also true but not too useful yet:
% If Conjecture \ref{conj:main_conj} holds for all $v_1, \dots, v_{2^{n-2}} \in \mathbb{F}_2^{n-1}$ satisfying the property that there are a multiple of $2^{t-1}$ copies of each $v_i$, then Conjecture \ref{conj:main_conj} holds for all $v_1, \dots, v_{2^{n-1}} \in \mathbb{F}_2^n$ satisfying the property that there are a multiple of $2^t$ copies of each $v_i$.
\end{proposition}
\begin{proof}
Assume without loss of generality that $v_1 = v_2 = 10\cdots 0$; we can change the basis to make $v_1$ = $v_2$ have this value. Also assume that $\sum_{i=2}^{2^{n-2}} v_{2i-1} \neq 0$; if it were 0 simply exchange $v_1$ and $v_2$ with any different pair. Define $v_i'$ to be $v_i$ with the first component switched to 0 if it was 1. Let $V'$ be the set containing $v_{2i-1}'$ for $2 \leq i \leq 2^{n-2}$, and the value $\sum_{i=2}^{2^{n-2}} v_{2i-1}'$. By assumption $V'$ contains $2^{n-2}$ nonzero elements in $\mathbb{F}_2^{n-1}$ summing to 0 (the vectors are technically in $\mathbb{F}_2^n$, but all have first component 0). Therefore, by assumption there exists a partition of $\mathbb{F}_2^{n-1}$ into pairs $(p_i', q_i')$ such that $p_i' + q_i'$ gives the $i$th element of $V'$, that is, $p_i' + q_i' = v_{2i-1}'$ for $2 \leq i \leq 2^{n-2}$ and $p_1' + q_1' = \sum_{i=2}^{2^{n-2}} v_{2i-1}'$.

We then construct a partition of $\mathbb{F}_2^n$ into pairs $(p_i, q_i)$ as follows. Let $p_1 = p_1'$ and $p_2 = q_1'$. If $v_{2i-1}$ has a 1 in the first component, let $p_{2i-1} = p_i'$ and $p_{2i} = q_i'$. Otherwise, let $p_{2i-1} = p_i'$ and $p_{2i} = p_i' + v_1$. Then, for all $1 \leq i \leq 2^{n-1}$, let $q_i = p_i + v_i$. It is easy to see that the four values $p_{2i-1}, p_{2i}, q_{2i-1}, q_{2i}$ are always equal to the four values $p_i', p_i' + v_1, q_i', q_i' + v_1$ in some order, so the pairs $(p_i, q_i)$ are a partition of $\mathbb{F}_2^n$.
\end{proof}
The following corollary directly follows from Proposition~\ref{lem:1_more_even}, as Conjecture~\ref{conj:main_conj} was verified to hold in $\mathbb{F}_2^5$ in \cite{balister_coloring_2011}.
\begin{corollary}
\label{cor:true_for_6}
For any $2^5$ vectors $v_1, \dots, v_{2^5} \in \mathbb{F}_2^6$ such that $v_{2i-1} = v_{2i}$ for all $1 \leq i \leq 2^4$, there exists a partition of $\mathbb{F}_2^6$ into pairs $(p_i, q_i)$ such that $p_i + q_i = v_i$ for all $1 \leq i \leq 2^5$.
\end{corollary}

% EXPLAIN PROPOSITION 2.7, KEEP 2.6 AS A LEMMA TOWARDS 2.7
% Finally, we show it for <=n distinct values without even condition

Below, we extend Lemma~\ref{lem:n_min_1_values} to the case where there are exactly $n$ different values among the $v_i$'s, which is difficult in that we cannot use our original induction because $\spn\{v_1, \dots, v_{2^{n-1}}\}$ can have dimension $n$. We use induction similar to that in Proposition~\ref{lem:1_more_even}, except that we maintain the property of each $v_i$ occurring an even number of times in the inductive step.
\begin{lemma}
\label{lem:n_values}
Let $v_1, \dots, v_{2^{n-1}} \in \mathbb{F}_2^n$ satisfy $v_{2i-1} = v_{2i}$ for all $1 \leq i \leq 2^{n-2}$ and let there be exactly $n$ distinct vectors $u_1, \dots, u_n$ among the $v_i$'s; each $v_i$ is equal to some value in $u_1, \dots, u_n$. Then there exists a partition of $\mathbb{F}_2^n$ into pairs $(p_i, q_i)$ such that $p_i + q_i = v_i$ for all $1 \leq i \leq 2^{n-1}$.
\end{lemma}
\begin{proof}
We show the result by induction, using the results shown in Corollary \ref{cor:true_for_6} and in Lemma \ref{lem:n_min_1_values} as base cases; the result holds for $n \leq 6$.
% (with a little extra work, the fact that Conjecture \ref{conj:main_conj} holds in $\mathbb{F}_2^5$ can be used instead of Corollary \ref{cor:true_for_6}).
For the inductive step, assume the result holds in $\mathbb{F}_2^{n-1}$. If $u_1, \dots, u_n$ are not all linearly independent, then we induct exactly as in Lemma \ref{lem:n_min_1_values}. Otherwise, assume that $u_1, \dots, u_n$ are all linearly independent. Assume without loss of generality that for $1 \leq i \leq n$, the vector $u_i$ has 1 at the $i$th component and 0 at all other components. This assumption is valid because $\{u_1, \dots u_n\}$ is a basis for $\mathbb{F}_2^n$, so we can simply switch to this basis. We now consider the following two cases, although the inductive step is nearly identical in each one:
\begin{description}
\item [$n \geq 8$.] It is easy to verify that one of $u_1, \dots, u_n$ must occur at least $2(n-1)$ times among $v_1, \dots v_{2^{n-1}}$ by the pigeonhole principle. Therefore assume without loss of generality that this value is $u_1$. Now construct the set $V' = \{v_1', \dots, v_{2^{n-2}}'\}$ as follows. First let $I_1, I_2, I_3$ be initialized as empty sets. For each $i$ such that $v_{2i-1}$ is not equal to $u_1$, let $v_i' = v_{2i-1}$, and place $i$ in $I_1$. Then for $2 \leq j \leq n$, if the number of copies of $u_j$ among $v_1, \dots, v_{2^{n-1}}$ is congruent to $2 \pmod{4}$, choose some $i$ not in $I_1$, let $v_i' = u_j$, and place $i$ in $I_2$ (such an $i$ always exists, as by assumption $|I_1| \leq 2^{n-1} - (n-1)$). Finally, for any $i$ not in $I_1$ or $I_2$, let $v_i' = u_2$ and place $i$ in $I_3$ (this decision is arbitrary; any value of $u_i$ for $i \geq 2$ would work).

By construction, each of $u_2, \dots, u_n$ must occur an even number of times in $V'$, and every element of $V'$ has a 0 at the first component, as $V'$ only consists of the values $u_2, \dots, u_n$. Therefore we can ignore the first component of the vectors in $V'$, and it follows by the inductive hypothesis there exists a partition of $\mathbb{F}_2^{n-1}$ into pairs $(p_i', q_i')$ such that $p_i' + q_i' = v_i'$ for $1 \leq i \leq 2^{n-2}$.

We then construct a partition of $\mathbb{F}_2^n$ into pairs $(p_i, q_i)$ as follows. For each $1 \leq i \leq 2^{n-2}$, if $i \in I_1$, let $p_{2i-1} = p_i'$ and $p_{2i} = p_i' + u_1$. Otherwise, $i \in I_2$ or $I \in I_3$, so $v_{2i-1} = v_{2i} = u_1$; in this case, let $p_{2i-1} = p_i'$ and $p_{2i} = q_i'$. Then let $q_i = p_i + v_i$ for all $1 \leq i \leq 2^{n-1}$. It is easy to verify that the four values $p_{2i-1}, p_{2i}, q_{2i-1}, q_{2i}$ are always equal to the four values $p_i', p_i' + u_1, q_i', q_i' + u_1$ in some order, so the pairs $(p_i, q_i)$ are a partition of $\mathbb{F}_2^n$.

\item [$n = 7$.] In this case, note that at least one of $u_1, \dots, u_7$ occurs at least 10 times among $v_1, \dots, v_{2^6}$, so we can assume without loss of generality that there at 10 copies of $u_1$. To use the same induction as above, note that if all of $u_2, \dots, u_7$ occured $2 \pmod{4}$ times among the $v_i$'s, then the total number of vectors $v_i$ would be congruent to $10 + 6 \cdot 2 \equiv 2 \pmod{4}$, which is a contradiction. Therefore at most 5 of the vectors $u_2, \dots, u_7$ occur $2 \pmod{4}$ times among the $v_i$'s, so the induction used for $n \geq 8$ still applies.
\end{description}
\end{proof}

% Note that the proof above is easily extended to having lower base cases, such as in $\mathbb{F}_2^5$ as opposed to $\mathbb{F}_2^6$, but more special cases are necessary similar to those for $n=7$.

The inductive step of the proof above is easily generalizable to any choice of the vectors $v_1, \dots, v_{2^{n-1}}$ with at most $2^{(n-2)/2}$ distinct values, assuming that $\sum_{i=1}^{2^{n-1}} v_i = 0$ and $v_{2i-1} = v_{2i}$ for all $1 \leq i \leq 2^{n-1}$ This is because it follows by the pigeonhole principle that some value must occur at least $2(x-1)$ times among the vectors $v_i$, where there are a total of $x$ distinct vectors among the $v_i$'s. However, it is very difficult to generalize Lemma \ref{lem:n_values} using this observation because the base case become much larger, as repeatedly applying the induction only reduces the number of distinct vectors among the $v_i$'s by 1 at each step. Therefore, for choices of $v_1, \dots, v_{2^{n-1}}$ with large $n$ and with approximately $2^{(n-2)/2}$ distinct values, the induction can only be applied a small number of times before the number of distinct values becomes too large to continue inducting.

Finally, we show below that Conjecture~\ref{conj:main_conj} holds for any choice of $v_1, \dots, v_{2^{n-1}} \in \mathbb{F}_2^n$ such that there are at most $n$ different values among the $v_i$'s. This result uses Lemmas \ref{lem:n_min_1_values} and \ref{lem:n_values}, but extends to the case where vectors do not necessarily occur an even number of times.
\begin{proposition}
\label{lem:n_vals_with_odd}
For any $2^{n-1}$ vectors $v_1, \dots, v_{2^{n-1}} \in \mathbb{F}_2^n$ satisfying $\sum_{i=1}^{2^{n-1}} v_i = 0$ and such that there are $l \leq n$ distinct vectors $u_1, \dots, u_l$ among the $v_i$'s, there exists a partition of $\mathbb{F}_2^n$ into pairs $(p_i, q_i)$ such that $p_i + q_i = v_i$ for all $1 \leq i \leq 2^{n-1}$.
\end{proposition}
\begin{proof}
We show the result by induction, using the fact that Conjecture \ref{conj:main_conj} holds in $\mathbb{F}_2^1, \dots, \mathbb{F}_2^5$ as well as Lemmas \ref{lem:n_min_1_values} and \ref{lem:n_values} and Corollary \ref{cor:true_for_6} as base cases. For the inductive step, assume the result holds in $\mathbb{F}_2^k$ for $2 \leq k \leq n - 1$, with $n \geq 6$. Furthermore, assume without loss of generality that $m$ is defined such that $u_1, \dots, u_m$ occur an odd number of times among $v_1, \dots, v_{2^{n-1}}$, and $u_{m+1}, \dots, u_l$ occur an even number of times. We can then assume that $m \geq 4$, as $m$ must be even, our base cases take care of $m = 0$, and $m$ cannot equal 2 because $\sum_{i=1}^m u_i = 0$ and all $u_i$ are distinct. Note that these assumptions imply that $\dim(\spn\{u_1, \dots, u_l\}) \leq l-1 \leq n -1$, as $u_1, \dots, u_l$ cannot be linearly independent. We consider the following cases independently, where each case assumes none of the previous ones were true.
\begin{description}
\item[$l < n$.] Partition $v_1, \dots, v_{2^{n-1}}$ into subsets $V_1$ and $V_2$ in the following manner. First place one copy of each vector $u_1, \dots, u_l$ into $V_1$. Then partition the remaining vectors among $V_1$ and $V_2$ in any way such that $|V_1| = |V_2| = 2^{n-2}$ and each vector occurs an even number of times in $V_2$. Then the result follows by applying the inductive hypothesis analogously to the way it is done in the proof of Lemma \ref{lem:n_min_1_values}.
\item[$m \leq n-2$.] We divide $v_1, \dots, v_{2^{n-1}}$ into subsets $V_1$ and $V_2$ as follows. First, place one copy of each vector $u_1, \dots, u_m$ in $V_1$. By the pigeonhole principle, there are at most $2^{n-2}$ copies of either $u_{n-1}$ or $u_n$ among $v_1, \dots, v_{2^{n-1}}$; assume without loss of generality that there at at most $2^{n-2}$ copies of $u_n$, and place all copies of $u_n$ in $V_2$. Then find the value $1 \leq i \leq n - 1$ for which $u_i$ has the fewest copies among $v_1, \dots, v_{2^{n-1}}$, and place all copies of $u_i$ in $V_1$. Then partition the remaining vectors among $V_1$ and $V_2$ in any way such that $|V_1| = |V_2| = 2^{n-2}$ and each vector occurs an even number of times in $V_2$. Then the result follows by applying the inductive hypothesis analogously to the way it is done in the proof of Lemma \ref{lem:n_min_1_values}.
%\end{description}

% Not true:
% Note that in the cases below we can assume that $n \geq 7$, as if $n = 6$ and $m > 4 = 6-2$, then $m = 6$, which implies that the total number of vectors is congruent to $6 \cdot (\pm 1) \equiv 2 \pmod{4}$, a contradiction.

%\begin{description}
\item[There exists a nonempty proper subset $U$ of $\{u_1, \dots, u_m\}$ with sum 0.] Partition \newline $v_1, \dots, v_{2^{n-1}}$ into subsets $V_1$ and $V_2$ as follows, noting that $|U| \geq 3$ and $$|\{u_1, \dots, u_m\} - U|  = m - |U| \geq 3.$$ If $|U|$ is even, place one copy of each element of $U$ in $V_1$, place one copy of each element of $\{u_1, \dots, u_m\} - U$ in $V_2$, then distribute the remaining vectors $v_i$ among $V_1$ and $V_2$ such that each subset receives an even number of additional copies of each vector, each subset has at most $n-1$ distinct vectors, and $|V_1| = |V_2| = 2^{n-2}$. It is easy to see such a partition exists; the distinctness condition is easily met using the pigeonhole principle to place all copies of some vector $v_{i_1}$ in $V_1$ and all copies of another vector $v_{i_2}$ in $V_2$.

Otherwise, if there exists no nonempty proper subset $U$ of $\{u_1, \dots, u_m\}$ with even size, then let $U_1 = U$ and $U_2 = \{u_1, \dots, u_m\} - U$. By assumption, $U_1$ and $U_2$ both have odd size and sum 0, but no nonempty proper subset of $U_1$ or of $U_2$ has sum 0; if there were such a subset, then there would exist a subset of $\{u_1, \dots, u_m\}$ of even size with sum 0. Furthermore, if there is some nonempty proper subset $U_3$ of $\{u_1, \dots, u_m\}$ and of odd size containing some elements of $U_1$ and some elements of $U_2$ with sum 0, then the symmetric difference $U_1 \oplus U_3$ would have sum 0 and even size, which is a contradiction. Therefore $\dim(\spn(U_1)) = |U_1| - 1$, $\dim(\spn(U_2)) = |U_2| - 1$, and $\dim(\spn(\{u_1, \dots, u_m\})) = m - 2$. It follows by the pigeonhole principle that there exists a basis for $\mathbb{F}_2^n$ for which some two vectors $u_i$ and $u_j$ with $i, j \leq m$ are prefixed with $01$ and all other vectors are prefixed with $00$, and such that the total number of copies of $u_i$ and $u_j$ combined is at most $2^{n-2}+2$. Then assume without loss of generality that $i = 1$ and $j = 2$, and proceed according to the case below, where no such set $U$ exists.
\item[No such set $U$ exists.] We give the proof for the case where $m = n$; the case where $m = n-1$ is nearly identical. Assume without loss of generality that $u_1 = 0100\dots0$, $u_2 = 0111\dots1$, all other $u_i$ are prefixed with $00$, and the number of copies of $u_1$ and $u_2$ combined is at most $2^{n-2} + 2$. Let the sets $V_1, V_2, V_3$ be constructed as follows. Define 
$$c = \begin{cases}
m/2+1 & \textrm{if } m \equiv 0 \pmod{4} \\
m/2 & \textrm{if } m \equiv 2 \pmod{4}. \\
\end{cases}$$
Begin by placing a single copy of $u_3, \dots, u_c$ in $V_1$ and a single copy of $u_{c+1}, \dots, u_m$ in $V_2$. After this step, $V_1$ and $V_2$ both have an odd number of elements. Then place the value $u_1' = \sum_{i=3}^c u_i$ in $V_1$, and the value $u_2' = \sum_{i=c+1}^m u_i$ in $V_2$, so that $V_1$ and $V_2$ both have sum 0. Finally, distribute all vectors in $\{v_1, \dots, v_{2^{n-1}}\} - \{u_1, \dots, u_m\}$ among $V_1, V_2, V_3$ such that
\begin{itemize}
\item $|V_1| = |V_2| = 2^{n-3}$ and $|V_3| = 2^{n-2}$,
\item All copies of $u_1$ and $u_2$ are in $V_3$,
\item There are at most $n-2$ distinct values in each $V_1$ and $V_2$ and at most $n-1$ distinct values in $V_3$,
\item Only $u_3, \dots, u_c$ occur an odd number of times in $V_1$, only $u_{c+1}, \dots, u_m$ occur an odd number of times in $V_2$, and no vector occurs an odd number of times in $V_3$.
\end{itemize}
By the pigeonhole principle, it must be possible to partition the vectors in this way. Note that all vectors in $V_1$ and $V_2$ are prefixed with $00$, and all vectors in $V_3$ are prefixed with $00$ or $01$. Then partition $\mathbb{F}_2^n$ into the affine subspaces $S_1$ consisting of all vectors prefixed with $10$, $S_2$ consisting of all vectors prefixed with $11$, and $S_3$ consisting of all vectors prefixed with $00$ or $01$. By the inductive hypothesis, for each $1 \leq i \leq 3$ there exists a partition of $S_i$ into pairs $(p_{i,j}', q_{i,j}')$ such that $p_{i,j}' + q_{i,j}' = v_{i,j}'$ for $1 \leq j \leq |V_i|$, where $v_{i,j}'$ denotes the $j$th element of $V_i$. Assume without loss of generality that $v_{1,1}' = u_1'$ and $v_{2,1}' = u_2'$. Then define the pairs $(p_{i,j}, q_{i,j})$ as follows. Let $$t = p_{1,1}' + p_{2,1}' + u_1,$$ so that $t$ is prefixed with $00$, and let $p_{1,1} = p_{1,1}'$, $q_{1,1} = p_{2,1}' + t$, $p_{2,1} = q_{1,1}'$, and $q_{2,1} = q_{2,1}' + t$. Then let 
\begin{alignat*}{3}
(p_{1,j}, q_{1,j}) &= (p_{1,j}', q_{1,j}') \hspace{.5cm} &&\textrm{ for } 2 \leq j \leq 2^{n-3} \\
(p_{2,j}, q_{2,j}) &= (p_{2,j}' + t, q_{2,j}' + t) \hspace{.5cm}&&\textrm{ for } 2 \leq j \leq 2^{n-3} \\
(p_{3,j}, q_{3,j}) &= (p_{3,j}', q_{3,j}') \hspace{.5cm} &&\textrm{ for }1 \leq j \leq 2^{n-2}.
\end{alignat*}
By this construction, the set of all sums $\{p_{i,j} + q_{i,j}\}$ for all possible $i$ and $j$ is equivalent to the set of all sums $\{p_{i,j}' + q_{i,j}'\}$ with $u_1 = p_{1,1} + q_{1,1}$ replacing $u_1' = p_{1,1}' + q_{1,1}'$ and $u_2 = p_{2,1} + q_{2,1}$ replacing $u_2' = p_{2,1}' + q_{2,1}'$. Therefore the set of all sums $\{p_{i,j} + q_{i,j}\}$ is equal to $\{v_1, \dots, v_{2^{n-1}}\}$, so the $(p_{i,j}, q_{i,j})$ form a valid partition of $\mathbb{F}_2^{n-1}$, completing the inductive step.
\end{description}
\end{proof}

The following theorem summarizes our progress on Conjecture \ref{conj:main_conj}. In Section \ref{sec:new_classes}, we apply the theorem to provide a method for inductively showing large classes of graphs to be set-sequential.
\begin{theorem}
\label{thm:all_conj_progress}
For any $2^{n-1}$ vectors $v_1, \dots, v_{2^{n-1}} \in \mathbb{F}_2^n$ such that $\sum_{i=1}^{2^{n-1}} v_i = 0$, if any of the following conditions is true, then there exists a partition of $\mathbb{F}_2^n$ into pairs $(p_i, q_i)$ such that $p_i + q_i = v_i$ for all $1 \leq i \leq 2^{n-1}$.
\begin{enumerate}
\item \label{case:dim_leq_5} $\dim(\spn\{v_1, \dots, v_{2^{n-1}}\}) \leq 5$.
\item \label{case:dim_6} $\dim(\spn\{v_1, \dots, v_{2^{n-1}}\}) = 6$ and each vector $v_i$ occurs an even number of times.
\item \label{case:n_distinct} There are at most $n$ distinct vectors in $\{v_1, \dots, v_{2^{n-1}}\}$.
\item \label{case:dim_n_div_2} $\dim(\spn\{v_1, \dots, v_{2^{n-1}}\}) \leq n/2$ and each vector $v_i$ occurs an even number of times.
\end{enumerate}
\end{theorem}
\begin{proof}
We cover the cases independently below.
\begin{enumerate}
\item This case is covered by Corollary \ref{cor:dim_5}.
\item This case is covered by Corollary \ref{cor:true_for_6} and Proposition~\ref{thm:small_dimension}, which also holds for the case where each vector occurs an even number of times.
\item This case is covered by Proposition~\ref{lem:n_vals_with_odd}.
\item This case is covered by Proposition \ref{lem:dim_span_n_div_2}.
\qedhere
\end{enumerate}
\end{proof}

\section{New classes of set-sequential trees}
\label{sec:new_classes}
In this section, we apply Theorem~\ref{thm:all_conj_progress} to show the set-sequentialness of many new classes of caterpillars. We then present a new inductive method for generating set-sequential trees.

To begin, we present the following lemma, which provides an application of the cases of Theorem \ref{thm:all_conj_progress} for which vectors can occur an odd number of times among $v_1, \dots, v_{2^{n-1}}$. Specifically, it provides a way of choosing such vectors $v_1, \dots, v_{2^{n-1}}$ from a set-sequential labeling of a graph $G$ such that $\sum_{i=1}^{2^{n-1}} v_i = 0$, where the choice of labels is based solely on the structure of $G$, and not based on the values of the labels. The lemma is very similar to results found in \cite{hegde_set_2009}, but we present the proof as it is relevant to future results in this paper.
\begin{lemma}
\label{lem:even_verts_sum_zero}
The sum of the labels of the even-degree vertices in any set-sequential labeling of a graph with at least two vertices is $0$.
\end{lemma}
\begin{proof}
For any set-sequential graph $G$, let $X =  \{x_1, \dots, x_{|V(G)|}\}$ be the set of labels of the vertices and let $Y = \{y_1, \dots, y_{|E(G)|}\}$ be the set of labels of the edges for some set-sequential labeling of $G$, and let $n = \log_2 (|V(G)| + |E(G)| + 1)$. Then $X \cap Y = \emptyset$ and $X \cup Y = \mathbb{F}_2^n - \{0\}$. Let $d_1, \dots, d_{|V(G)|}$ be the degrees$\pmod{2}$ of the vertices of $G$. Then $$0 = \sum_{v \in \mathbb{F}_2^n - \{0\}} v = \sum_{v \in X \cup Y} v = \sum_{i=1}^{|V(G)|} x_i + \sum_{i=1}^{|E(G)|} y_i = \sum_{i=1}^{|V(G)|} x_i + \sum_{i=1}^{|V(G)|} d_i x_i = \sum_{i=1}^{|V(G)|} (d_i+1) x_i = \sum_{i | d_i = 0} x_i.$$
\end{proof}

The following result applies Theorem~\ref{thm:all_conj_progress} and Lemma~\ref{lem:even_verts_sum_zero}, in order to present a general method for inductively generating set-sequential trees.
\begin{theorem}
\label{cor:induction_using_conj_progress}
For any $2^{n-1}$-vertex set-sequential graph $G$, let $W = \{w_1, \dots, w_k\}$ denote some $k$-element subset of the vertices of $G$. For any sequence $c_1, \dots, c_k$ such that $\sum_{i=1}^k c_i = 2^{n-1}$, let $G'$ be the $2^n$-vertex graph obtained from $G$ by adding $c_i$ pendant edges to $w_i$ for all $1 \leq i \leq k$. Suppose that $G'$ has no even-degree vertices and one of the following conditions hold:
\begin{enumerate}
\item $\dim(\spn(W)) \leq 5$.
\item $\dim(\spn(W)) = 6$ and all $c_i$ are even.
\item $|W| = k \leq n$.
\item $\dim(\spn(W)) \leq n/2$ and all $c_i$ are even.
\end{enumerate}
Then the graph $G'$ is set-sequential. Furthermore, if $W'$ gives the vertices in $G'$ corresponding to $W$ from $G$, then $\dim(\spn(W')) = \dim(\spn(W))$.
\end{theorem}
\begin{proof}
Let $v_1, \dots, v_{2^{n-1}}$ consist of $c_i$ copies of $w_i$ for $1 \leq i \leq k$, so that the $i$th new pendant edge in $G'$ is attached to a vertex labeled $v_i$ for $1 \leq i \leq 2^{n-1}$. Let $d_i$ be the degree of $w_i$ in $G$. By Lemma \ref{lem:even_verts_sum_zero} and because $G'$ has no even-degree vertices, $$\sum_{i=1}^{2^{n-1}} v_i = \sum_{i=1}^k c_i w_i = \sum_{i=1}^k (d_i+1) w_i = \sum_{i|d_i = 0} w_i = \sum_{v \in V(G) \textrm{ of even degree}} v = 0.$$ Therefore, by the definition of the vectors $v_i$, Theorem \ref{thm:all_conj_progress} implies that there exists a partition of $\mathbb{F}_2^n$ into pairs $(p_i, q_i)$ such that $p_i + q_i = v_i$ for all $1 \leq i \leq 2^{n-1}$. Then append $0$ to the labels of all original vertices and edges in $G'$, or those that were also in $G$, and label each new pendant vertex and edge in $G'$ with $1p_i$ and $1q_i$ respectively, where $xy$ denotes $x$ concatenated with $y$. Then the labeling of $G'$ is set-sequential, and $\dim(\spn(W')) = \dim(\spn(W))$.
\end{proof}
Note that in Theorem~\ref{cor:induction_using_conj_progress}, the graph $G$ can have even-degree vertices, although $G'$ cannot. This property is useful in the next section, where large classes of caterpillars with only odd-degree vertices are shown to be set-sequential based on the fact that a few caterpillars having even-degree vertices are set-sequential.

\subsection{Caterpillars with only odd-degree vertices}
\label{subsec:app_conj}
In this section, we show that any caterpillar $T$ of diameter $k$ with only odd-degree vertices such that $k \leq 18$ or $2^{k-1} \leq |V(T)|$ is set-sequential.

Here, we only discuss trees with $2^{n-1}$ vertices for some positive integer $n$, unless explicitly stated otherwise. Abhishek and Agustine \cite{abhishek_set-valued_2012} showed that a caterpillar with diameter 4 is set-sequential if and only if each vertex has odd degree. Abhishek \cite{abhishek_set-valued_2013} extended this result by showing the same characterization holds for set-sequential caterpillars of diameter~5.

The motivation behind Conjecture \ref{conj:main_conj} presented in \cite{balister_coloring_2011} is easily applied to show that the conjecture, if true, would imply the existence of a finite set of caterpillars of some given diameter $k$ for which if these caterpillars are set-sequential, then all caterpillars of diameter $k$ with only odd-degree vertices are set-sequential. This statement is not extended to caterpillars with even-degree vertices because infinite classes of such caterpillars that are not set-sequential are shown to exist in \cite{hegde_set_2009}. The cases of the conjecture that we prove above are sufficient to show that such a finite set exists, although this finite set would be significantly smaller if the entire conjecture were known to be true. However, for sufficiently small $n$, we verify below that all graphs in this finite set are set-sequential.

%To begin, the following lemma is well-known, but we present an alternative proof using Theorem \ref{thm:all_conj_progress} that is generalizable to more caterpillars. (Remember that we are only considering trees with $2^{n-1}$ vertices for some $n$.)
%\begin{lemma}
%\label{lem:stars}
%All stars are set-sequential.
%\end{lemma}
%\begin{proof}
%We show the result using induction. For the base case, it is easy to see that the graph consisting of a single edge is set-sequential. For the inductive step, let $T$ denote the $2^{n-1}$-vertex star, and assume that $T$ is set-sequential. Let $v_1 = v_2 = \cdots = v_{2^{n-1}}$ be equal to the label of the internal node for the set-sequential labelling of $T$. Then, by Case 1 in Theorem \ref{thm:all_conj_progress}, there exists a partition of $\mathbb{F}_2^n$ into pairs $(p_i, q_i)$ such that $p_i + q_i = v_i$ for all $1 \leq i \leq 2^{n-1}$. Define the graph $T'$ to be $T$ with $2^{n-1}$ additional pendant edges, where $0$ is appended to the label of each original vertex and edge in $T$, and the $i$th new pendant vertex and pendant edge in $T'$ are labelled $1p_i$ and $1q_i$ respectively. Then $T'$ is set-sequential.
%\end{proof}
%As a star is a caterpillar with diameter at most 2, we use Lemma \ref{lem:stars} as a base case for induction.
Following the notation of Abhishek and Agustine \cite{abhishek_set-valued_2012}, let $T[d_1, \dots, d_k]$ denote the caterpillar with a $k$-vertex center path such that the $i$th center path vertex has degree $d_i$. This notation gives some freedom in expressing caterpillars; we typically use the unique form such that $d_i \geq 2$ for all $1 \leq i \leq k$, so that the caterpillar has diameter $k+1$. As a necessary exception, we use $T[1]$ to denote the graph consisting of a single edge. In the result below, we apply the first two cases in Theorem~\ref{cor:induction_using_conj_progress} in order to show that all caterpillars with diameter at most 18 are set-sequential.
\begin{theorem}
\label{thm:small_k}
All caterpillars with only odd-degree vertices and diameter $k \leq 18$ are set-sequential.
\end{theorem}
\begin{proof}
We show the result by induction. For the base case, all of the caterpillars below are set-sequential. This is easily verified with a randomized greedy computer search.
\begin{enumerate}
\item $T[1]$
\item $T[5,3,3,3,3,3]$, $T[3,5,3,3,3,3]$, $T[3,3,5,3,3,3]$
\item $T[3,3,3,3,3,3,3]$
\item $T[3,3,3,2,2,2,2,2,2,3]$
%\item $T[3,3,2,2,2,2,2,2,2,2,3]$
%\item $T[3,2,2,2,2,2,2,2,2,2,2,3]$
\item $T[3,2,2,2,2,2,2,2,2,2,2,2,2]$
\item $T[2,2,2,2,2,2,2,2,2,2,2,2,2,2]$
\end{enumerate}

For the inductive step, we repeatedly apply Theorem~\ref{cor:induction_using_conj_progress} starting from the base cases above following the order below, where $G_1 \Rightarrow G_2$ means that the set-sequentialness of $G_1$ implies the set-sequentialness of $G_2$ under Theorem~\ref{cor:induction_using_conj_progress}. Note that because $\dim(\spn(W')) = \dim(\spn(W))$ in Theorem~\ref{cor:induction_using_conj_progress}, the dimension of the span of the labels of the center path vertices, which represent the set $W$ in Theorem~\ref{cor:induction_using_conj_progress}, remains at most 6 under the induction below, even though the number of vertices in the caterpillars grows arbitrarily large.
\begin{itemize}
\item $T[1] \Rightarrow$ All caterpillars with no even-degree vertices of diameter at most 2; includes~$T[2]$.
\item $T[2] \Rightarrow$ All caterpillars with no even-degree vertices of diameter 3 or 4; includes~$T[3, 3, 3]$.
\item $T[3,3,3] \Rightarrow$ All caterpillars with no even-degree vertices of diameter 5 or 6.
\item $T[5,3,3,3,3,3] \Rightarrow$ All caterpillars with no even-degree vertices of diameter 7 of the form $T[d_1,d_2,d_3,d_4,d_5,d_6]$ such that $d_1 > 3$ or $d_6 > 3$. Note that the second condition follows by symmetry.
\item $T[3,5,3,3,3,3] \Rightarrow$ All caterpillars with no even-degree vertices of diameter 7 of the form $T[d_1,d_2,d_3,d_4,d_5,d_6]$ such that $d_2 > 3$ or $d_5 > 3$. Note that the second condition follows by symmetry.
\item $T[3,3,5,3,3,3] \Rightarrow$ All caterpillars with no even-degree vertices of diameter 7 of the form $T[d_1,d_2,d_3,d_4,d_5,d_6]$ such that $d_3 > 3$ or $d_4 > 3$. Note that the second condition follows by symmetry.
\item $T[3,3,3,3,3,3,3] \Rightarrow$ All caterpillars with no even-degree vertices of diameter 8, 9, or~10.
\item $T[3,3,3,2,2,2,2,2,2,3] \Rightarrow$ All caterpillars with no even-degree vertices of diameter 11, 12, or 13.
\item $T[3,2,2,2,2,2,2,2,2,2,2,2,2] \Rightarrow$ All caterpillars with no even-degree vertices of diameter 14.
\item $T[2,2,2,2,2,2,2,2,2,2,2,2,2,2] \Rightarrow$ All caterpillars with no even-degree vertices of diameter 15, 16, or 17; includes $T[3,3,3,3,3,3,3,3,3,3,3,3,3,3,3]$.
\item $T[3,3,3,3,3,3,3,3,3,3,3,3,3,3,3] \Rightarrow$ All caterpillars with no even-degree vertices of diameter 18.
\end{itemize}

To illustrate the validity of the assertions above, we explain how the set-sequentialness of $T[3,3,3,2,2,2,2,2,2,3]$ implies the set-sequentialness of all caterpillars with no even-degree vertices of diameter 11, 12, or 13; an analagous argument applies to the other statements. Let $T_0 = T[3,3,3,2,2,2,2,2,2,3]$, and let $\mathcal{F}$ denote the family of caterpillars of diameter 11, 12, or 13 with no even-degree vertices. Because $|V(T_0)| = 16$, the set-sequential labeling of $T_0$ consists of vectors in $\mathbb{F}_2^5$, so by Theorem~\ref{cor:induction_using_conj_progress}, any 32-vertex tree $T'$ with no even-degree vertices consisting of $T_0$ with 16 added pendant edges is set-sequential. Therefore, all 32-vertex caterpillars in $\mathcal{F}$ are set-sequential. Furthermore, as all of these 32-vertex caterpillars can be constructed by adding pendant edges to only the 12 center path vertices of $T_0$, where $T_0$ is considered in the equivalent form $T[1,3,3,3,2,2,2,2,2,2,3,1]$, so that the center path has 12 vertices as opposed to 10, it follows that if $W'$ denotes the set of labels of center path vertices in some 32-vertex caterpillar in $\mathcal{F}$, then $\dim(\spn(W')) \leq 5$ by Theorem~\ref{cor:induction_using_conj_progress}.

From here, we proceed through induction. For some $n \geq 6$, assume that all $2^{n-1}$-vertex caterpillars in $\mathcal{F}$ are set-sequential, and the dimension of the span of the center path vertices is at most 5 in all of the set-sequential labelings. Then for any $2^n$-vertex caterpillar $T'$ in $\mathcal{F}$, there exists a $2^{n-1}$-vertex subgraph $T$ of $T'$ in $\mathcal{F}$ such that $T$ and $T'$ share the same center path; $T'$ consists of $T$ with $2^{n-1}$ pendant edges added to its center path. Therefore, by Theorem~\ref{cor:induction_using_conj_progress}, there exists a set-sequential labeling of $T'$ such that the dimension of the span of the labels of the center path vertices is at most 5, completing the inductive step.

%For the inductive step, note that if there exists some graph $G'$ that can be obtained by adding $2^{n-1}$ pendant edges to a set-sequential $2^{n-1}$-vertex graph $G$, where the dimension of the span of the labels of the vertices from which the pendant edges stem is at most 6, and an even number of new pendant edges stem from each vertex, then $G$ is also set-sequential by Theorem \ref{thm:small_dimension} and Corollary \ref{cor:true_for_6}. By construction, all of the graphs in the base cases have at most 32 vertices, and therefore are labelled with elements of $\mathbb{F}_2^n$ for some $n \leq 6$. Every time a $2^{n-1}$ vertex graph is transformed into a $2^n$ vertex graph in the manner described above, all new edges and vertices are given a leading component with value 1, and all old edges and vertices are given a leading component with value 0. Therefore the dimension of the span of the labels of the vertices on the path of the caterpillar remains constant as pendant edges are added. Because each caterpillar with path length at most 17 can be obtained by repeatedly adding pendant edges to the path vertices of one of the graphs from the base cases in the manner described above, the inductive step holds.
\end{proof}

Because the result above is ultimately dependent on the fact that Conjecture \ref{conj:main_conj} has been verified for $n \leq 5$, in order to extend the upper bound of 18 on the caterpillar diameter, it would be necessary to either verify that the conjecture holds for larger $n$, or to verify that many more caterpillars serving as base cases are set-sequential. However, by applying Case \ref{case:n_distinct} of Theorem~\ref{cor:induction_using_conj_progress}, we show below that all caterpillars with sufficiently many vertices, and no even-degree vertices, of any given diameter are set-sequential.

\begin{theorem}
\label{thm:large_caterpillars}
For all $k \geq 0$, all caterpillars of diameter $k+1$ that have only odd-degree vertices and have at least $2^k$ vertices, with $2^n$ total vertices for some $n$, are set-sequential.
\end{theorem}
\begin{proof}
We show the result by induction. For the base case, note that the result holds for $k \leq 1$ by Theorem \ref{thm:small_k}. For the inductive step, let $k \geq 2$, and assume that the result holds for all caterpillars of diameter at most $k$. We now show that any caterpillar $T' = T[d_1, \dots, d_k]$ of diameter $k+1$ with $2^n$ vertices is set-sequential if $2^n \geq 2^k$. First note that $T'$ has $2^n - k \geq 2^{n-1} + (k-2)$ pendant edges, so it is possible to remove $2^{n-1}$ pendant edges from $T'$ while maintaining the property that each vertex has odd degree. By assuming without loss of generality that $d_1 \leq d_k$, then because $(d_1-1) + (d_k-1) < |E(T')| = 2^n-1$, it follows that $d_1-1 < 2^{n-1}$. Therefore, there exists a graph $T$ obtained by removing all $d_1-1$ pendant edges stemming from the first vertex in the path of $T'$, and some additional pendant edges, such that $T$ has $2^{n-1}$ vertices, all with odd degree. Because vertices in the path have odd degree before and after the removal, each of these vertices lost an even number of pendant edges. Furthermore, by construction $T$ has diameter either $k$ or $k-1$, and $|V(T)| = 2^{n-1} \geq 2^{k-1} \geq 2^{k-2}$, so therefore $T$ is set-sequential by the inductive hypothesis. It follows that $T'$ is set-sequential by Case \ref{case:n_distinct} of Theorem~\ref{cor:induction_using_conj_progress}.
\end{proof}

The argument used to show Theorem \ref{thm:large_caterpillars} can be extended to other general classes of trees. Intuitively, it can be used to show the set-sequentialness of trees with sufficiently many pendant edges, or vertices with degree 1.

%For example, it is easy to see from the proof that the following generalization holds.
%
%\begin{theorem}
%For any tree $T$, if sufficiently many pendant edges are added to the vertices of $T$ in any way such that the resulting tree $T'$ has no even-degree vertices (and such that $|V(T')|$ is a power of $2$), then $T'$ is set-sequential.
%\end{theorem}
%\begin{proof}
%DO THIS
%\end{proof}

\subsection{Attaching copies of set-sequential trees}
\label{subsec:copies_induction}
Although Theorem~\ref{cor:induction_using_conj_progress} provides a powerful way to inductively produce set-sequential trees, it is most useful for showing that trees with many vertices for a fixed diameter are set-sequential. Likewise, this method is not very useful for finding trees with relatively large diameters and few vertices, such as trees consisting of a long path with a few small offshoots. However, these types of trees containing long paths can serve as base cases in induction using Theorem~\ref{cor:induction_using_conj_progress}, and are therefore of interest. Mehta and Vijayakumar \cite{mehta_note_2008} show that all paths with at least 16 vertices are set-sequential. Below, we use a similar method of proof as in \cite{mehta_note_2008}, but we show a much more general result, which applies to all trees as opposed to only paths.

Let $Z$ be a set-sequential path with vertices $u$ and $v$ of degree 1, and let $k$ be defined so that $Z$ has $(k+1)$/2 vertices. Let $z_1, z_3, z_5, \dots, z_{k-2}, z_k \in \mathbb{F}_2^n$ be the labels of the vertices of $Z$ in order from $u$ to $v$, and let $z_2, z_4, z_6, \dots, z_{k-3}, z_{k-1} \in \mathbb{F}_2^n$ be the labels of the edges of $Z$ in the same order, so that $z_1 = u$, $z_k = v$, and $z_{2i-1} + z_{2i+1} = z_{2i}$ for all $1 \leq i \leq (k-1)/2$. In \cite{mehta_note_2008}, it was shown that there exist sequences $a_i$, $b_i$, $c_i$, and $d_i$ of vectors in $\mathbb{F}_2^2$ defined over all integer $i$ such that the sequence $w_i \in \mathbb{F}_2^{n+2}$ for $1 \leq i \leq 4k+3$ given by Table \ref{tab:w_i_def} satisfies $w_{2i-1} + w_{2i+1} = w_{2i}$ for all $1 \leq i \leq 2k+1$, and all $w_i$ are distinct. Using this construction, we show the following result.

\begin{table}
\centering
\begin{tabular}{c c c c c c c c c c c c c c c c c c c}
\hhline{===============}
\bf Value & $w_1$ & $\dots$ & $w_k$ & $w_{k+1}$ & $w_{k+2}$ & $\dots$ & $w_{2k-1}$ & $w_{2k}$ & $w_{2k+1}$ \\
\hline
\bf Prefix & $a_k$ & $\dots$ & $a_1$ & $10$ & $b_1$ & $\dots$ & $b_{k-2}$ & $b_k$ & $b_{k-1}$ \\
\bf Suffix & $z_k$ & $\dots$ & $z_1$ & $0_n$ & $z_1$ & $\dots$ & $z_{k-2}$ & $z_k$ & $z_{k-1}$ \\
\hhline{===============}
\bf Value & $w_{2k+2}$ & $w_{2k+3}$ & $w_{2k+4}$ & $w_{2k+5}$ & $\dots$ & $w_{3k}$ & $w_{3k+1}$ & $w_{3k+2}$ & $w_{3k+3}$ \\
\hline
\bf Prefix & $11$ & $c_{k-1}$ & $c_k$ & $c_{k-2}$ & $\dots$ & $c_3$ & $c_1$ & $c_2$ & $01$ \\
\bf Suffix & $0_n$ & $z_{k-1}$ & $z_k$ & $z_{k-2}$ & $\dots$ & $z_3$ & $z_1$ & $z_2$ & $0_n$ \\
\hhline{===============}
\bf Value & $w_{3k+4}$ & $w_{3k+5}$ & $w_{3k+6}$ & $\dots$ & $w_{4k+3}$ \\
\hline
\bf Prefix & $d_2$ & $d_1$ & $d_3$ & $\dots$ & $d_k$ \\
\bf Suffix & $z_2$ & $z_1$ & $z_3$ & $\dots$ & $z_k$ \\
\hhline{===============}
\end{tabular}
\caption{The definition of the sequence $w_1, \dots, w_{4k+3}$ given in \cite{mehta_note_2008}, where each value $w_i$ consists of the prefix concatenated with the suffix. Above, $0_n$ denotes the zero vector in $\mathbb{F}_2^n$, and the notation $x_{i_1} \dots x_{i_2}$ denotes the sequence containing $x_i$ for all $i$ between $i_1$ and $i_2$, in increasing order if $i_1 \leq i_2$ and in decreasing order if $i_1 > i_2$.}
\label{tab:w_i_def}
\end{table}

\begin{theorem}
\label{thm:4_copies}
For any set-sequential tree $T$ with at least 3 vertices, let $u$ and $v$ be any distinct vertices in $T$ with degree 1. Let $u_1, \dots, u_4$ and $v_1, \dots, v_4$ be the vertices corresponding to $u$ and $v$ respectively in 4 distinct copies $T_1, \dots, T_4$ of $T$. Then the tree $T'$ obtained by adding $(u_1, u_2)$, $(v_2, v_3)$, and $(u_3, u_4)$ as edges to the union of $T_1$, $T_2$, $T_3$, and $T_4$ is set-sequential.
\end{theorem}
\begin{proof}
Let $Z$ be the unique path in $T$ connecting $u$ and $v$, so that when the vectors $z_i \in \mathbb{F}_2^n$ are defined as above, $z_1 = u$ and $z_k = v$. We observe that most of the properties of the sequence $w_i$ constructed from the path $Z$ are not dependent on the length of $Z$. Specifically, for any odd $k \geq 5$, so that $|V(Z)| \geq 3$, the sequence $w_i$ maintains the properties that $w_{2i-1} + w_{2i+1} = w_{2i}$ for all $1 \leq i \leq 2k+1$, and all $w_i$ are distinct. Therefore, the sequence $w_i$ serves as a partial set-sequential labeling for the tree $T'$, where the vertices of the path $W$ connecting $v_1$ and $v_4$ in $T'$ are labeled with $w_1, w_3, w_5, \dots, w_{4k+1}, w_{4k+3}$ in order from $v_1$ to $v_4$, and the edges of $W$ are labeled $w_2, w_4, w_6, \dots, w_{4k}, w_{4k+2}$.

It remains to complete the set-sequential labeling of $T'$ by providing labels for all vertices not in $W$. Note that in the sequence $w_i$ for $1 \leq i \leq 4k+3$, each vector $z_i$ for $1 \leq i \leq k$ occurs exactly 4 times as an $n$-dimensional suffix, once with each of the 4 possible 2-dimensional prefixes. The remaining 3 vectors among the $w_i$ consist of the $n$-dimensional $0$ vector as a suffix, concatenated with each of the 3 possible nonzero 2-dimensional prefixes. Therefore, for each label $x$ of a vertex or edge in $T$ but not in $Z$, if the labels of the 4 copies of that vertex or edge in $T'$ are $00x$, $01x$, $10x$, and $11x$ in some order, then all $2^{n+2}-1$ labels in the labeling of $T'$ are guaranteed to be distinct. Using this idea, we present the following method of labeling the remaining vertices and edges in $T'$.

Let $q$ and $r$ be any two connected vertices in $T$ with distances $d-1$ and $d$ from the path $Z$ respectively. Let $q_1, \dots, q_4$ and $r_1, \dots, r_4$ be the copies of $q$ and $r$ in $T_1, \dots, T_4$ respectively. We inductively provide labels for $r_1, \dots, r_4$ given that $q_1, \dots, q_4$ are labeled with $00q$, $01q$, $10q$, $11q$ in some order, using the existing labels of $W$, which correspond to a distance 0 from the path, as base cases. The base cases rely on the property of the sequence $w_i$ that for any vertex label $z_i$ other than $u$ or $v$, the 4 corresponding copies of the vertex in $T'$ are labeled with all 4 possible 2-dimensional prefixes concatenated to $z_i$. For the inductive step, if $p_i$ is the 2-dimensional prefix of the label of $q_i$, then label $r_i$ with $f(p_i)$ concatenated to $r$, where $f$ is defined by 
\begin{align}
\nonumber
\begin{split}
f(00) &= 00 \\
f(01) &= 10 \\
f(10) &= 11 \\
f(11) &= 01. \\
\end{split}
\end{align}
Therefore, the edge $e_i$ connecting $q_i$ and $r_i$ (corresponding to the edge $e$ in $T$, connecting $q$ and $r$), is labeled with $$p_iq + f(p_i)r = (p_i + f(p_i))(q + r) = (p_i + f(p_i))e.$$ Because both the functions $f(p)$ and $p + f(p)$ are one-to-one in $\mathbb{F}_2^2$, it follows that the labels of the 4 copies of both $r$ and $e$ in $T'$ consist of all 4 2-dimensional prefixes concatenated with $r$ and $e$ respectively. Therefore all labels in $T'$ are distinct, so the labeling is set-sequential.
\end{proof}

For a set-sequential tree $T$ with $x$ odd-degree vertices and $y$ even-degree vertices, Theorem~\ref{thm:4_copies} shows the set-sequentialness of at least one tree $T'$ with $4x-6$ odd-degree vertices and $4y+6$ even-degree vertices. However, if $y=0$, Case \ref{case:dim_leq_5} of Theorem~\ref{cor:induction_using_conj_progress} is easily used to create a tree $T''$ with twice the number of vertices as $T'$ containing no even-degree vertices. Furthermore, Case \ref{case:n_distinct} of the corollary can be used to create such a tree $T''$ even if $T'$ has more than 6 even-degree vertices, as long as $T'$ has sufficiently many vertices. Using this idea, it is easy to see that Theorem \ref{thm:4_copies} can be used to construct sequences of trees, which may be caterpillars, such that the diameter grows nearly linearly with the number of vertices. Specifically, such sequences would be constructed by repeatedly applying Theorem \ref{thm:4_copies} $t$ times and Theorem~\ref{cor:induction_using_conj_progress} one time for some integer $t$, beginning from some base tree.

Another motivation for Theorem \ref{thm:4_copies} was to find a caterpillar $T_k$ of diameter $k$ with at most $2k$ vertices such that $T$ is set-sequential for as many positive integers $k$ as possible. Each such set-sequential caterpillar $T_k$ could serve as a base case for showing the set-sequentialness of all caterpillars of diameter $k$ with no even-degree vertices through induction if Conjecture~\ref{conj:main_conj} were known to be true, or if a sufficiently strong subset of the conjecture were known to hold. Theorem \ref{thm:4_copies} succeeds in showing the existence of infinitely many such caterpillars $T_k$ other than paths, which were shown to be set-sequential in \cite{mehta_note_2008}. For example, the theorem can be repeatedly applied beginning with the star with 3 leaves in order to show that there exists a set-sequential caterpillar $T_k$ of diameter $k$ with at most $2k$ vertices for all integers $k$ that can be written in the form $k = 3 \cdot 4^c - 1$ for some nonnegative integer $c$. %ADD CONJECTURE ABOUT WHICH $T_k$ EXIST?

\section{Concluding Remarks}
\label{sec:conclusion}
In this paper, we investigate the problem of classifying set-sequential trees. We resolve many restricted cases of Conjecture \ref{conj:main_conj}, which Balister, Gy\H{o}ri, and Schelp \cite{balister_coloring_2011} introduced as a potential general method for showing trees to be set-sequential.
% a conjecture introduced by Balister, Gy\H{o}ri, and Schelp \cite{balister_coloring_2011} as a potential method of showing the set-sequentialness of many general classes of trees. 
We then apply our progress on Conjecture \ref{conj:main_conj} to show the set-sequentialness of many new classes of caterpillars. We also introduce a new method for constructing set-sequential trees that can be applied to many classes of trees for which the method of \cite{balister_coloring_2011} is not useful.

Many of our results are motivated by Conjecture \ref{conj:odd_deg_trees}, which states that all trees with only odd-degree vertices, and with $2^n$ vertices for some nonnegative integer $n$, are set-sequential. In Section \ref{subsec:app_conj}, we resolve certain cases of this conjecture by applying the idea introduced in \cite{balister_coloring_2011} of generating set-sequential trees by adding $2^{n-1}$ pendant edges to a $2^{n-1}$-vertex tree. However, this method cannot prove a tree to be set-sequential for which fewer than half of the vertices have degree 1. With respect to this limitation, Theorem \ref{thm:4_copies} is particularly interesting, as it shows the set-sequentialness of classes of trees with arbitrarily few vertices of degree 1. We suggest that additional similar results could help approach the problem of showing the set-sequentialness of all trees with only odd-degree vertices.

\section{Acknowledgements}
We would like to thank Dr. Tanya Khovanova for her helpful comments. We would also like to thank the MIT PRIMES program for the opportunity to perform this research. This material is based upon work supported by the National Science Foundation under Grant no.~DMS-1519580.

%\newpage
\bibliographystyle{ieeetr}
\singlespacing
\bibliography{library.bib,library_manual.bib}
\end{document}